\documentclass[11pt, reqno]{amsart}
\usepackage{amscd}        % Package used to produce simple commutative diagrams

\usepackage{float}

\usepackage{tabu}

\usepackage{nomencl, url}

\usepackage{algorithm, verbatim}
\usepackage{algpseudocode}

\usepackage{multirow}

\usepackage{tikz-cd}
\usepackage{graphicx}
\usepackage{rotating}
\usepackage{diagbox}
\usepackage{amssymb}
\usepackage{epstopdf}
\usepackage{tikz}
\definecolor{mintgreen}{RGB}{152,255,152}
\definecolor{pinksalmon}{RGB}{255,102,102}
\definecolor{hueso}{RGB}{245,245,220}
\definecolor{marfil}{RGB}{255,253,208}
\definecolor{amarillo}{RGB}{255,255,0}
\usetikzlibrary{decorations.markings,arrows}
\usetikzlibrary{decorations.pathreplacing}
\usepackage[inner=1.0in,outer=1.0in,bottom=1.0in, top=1.0in]{geometry}

\numberwithin{equation}{section}
\newtheorem{theorem}{Theorem}[section]

\newtheorem{proposition}[theorem]{Proposition}
\newtheorem{corollary}[theorem]{Corollary}

\newtheorem{hypothesis}[theorem]{Hypothesis}

\makeatletter
\def\moverlay{\mathpalette\mov@rlay}
\def\mov@rlay#1#2{\leavevmode\vtop{%
   \baselineskip\z@skip \lineskiplimit-\maxdimen
   \ialign{\hfil$\m@th#1##$\hfil\cr#2\crcr}}}
\newcommand{\charfusion}[3][\mathord]{
    #1{\ifx#1\mathop\vphantom{#2}\fi
        \mathpalette\mov@rlay{#2\cr#3}
      }
    \ifx#1\mathop\expandafter\displaylimits\fi}
\makeatother

\newcommand{\suchthat}{\;\ifnum\currentgrouptype=16 \middle\fi|\;}

\newcommand{\C}{\mathbb{C}}
\newcommand{\Q}{\mathbb{Q}}
\newcommand{\R}{\mathbb{R}}
\newcommand{\Gal}[1]{\operatorname{Gal}#1}
\newcommand{\op}[1]{\operatorname{#1}}

\newcommand{\N}{{\rm{N}}}

\newcommand{\calN}{\mathcal{N}}

\newcommand{\mfc}{\mathfrak{c}}

\newcommand{\beq}{\begin{equation}}
\newcommand{\eeq}{\end{equation}}

\newcommand\cm{\operatorname{cm}}
\newcommand\Weyl{\operatorname{Weyl}}
\newcommand\Res{\operatorname{Res}}

\newcommand\Cl{\operatorname{Cl}}

\renewcommand\b\bullet
\renewcommand\c\times

\theoremstyle{definition}
\newtheorem{remark}[theorem]{Remark}

\newtheorem{example}[theorem]{Example}

\def\C{{\mathbb C}}
\def\R{{\mathbb R}}

\def\Q{{\mathbb Q}}

\def\b{{\frak b}}
\def\O_K{{\Cal{O}_{K}}}
\def\O_F{{\Cal{O}_{F}}}
\def\N_F{{\Cal{N}_{F/\Q}}}

\begin{document}

\title{The distribution of $G$-Weyl CM fields and the Colmez conjecture}
\author{}
%\date{}                                           % Activate to display a given date or no date

\author{Adrian Barquero-Sanchez, Riad Masri, and Frank Thorne}

\address{Escuela de Matem\'atica, Universidad de Costa Rica, San Jos\'e 11501, Costa Rica}

\email{adrian.barquero\_s@ucr.ac.cr}

\address{Department of Mathematics, Mailstop 3368, Texas A\&M University, College Station, TX 77843-3368 }

\email{masri@math.tamu.edu}

\address{University of South Carolina, Department of Mathematics, 317-O LeConte College, 1523 Greene Street, Columbia, SC 29201}

\email{thorne@math.sc.edu}

\begin{abstract} Let $G$ be a transitive subgroup of $S_d$ and $E$ be a CM field of degree $2d$ 
with a maximal totally real $G$-field. If the Galois group of the Galois closure of $E$ is isomorphic to the wreath product of 
$C_2$ and $G$, then we say that $E$ is a \textit{$G$-Weyl} CM field. 

Let $N_{2d}^{\textrm{Weyl}}(X,G)$ count the $G$-Weyl CM fields $E$ of degree $2d$ with discriminant $|d_E| \leq X$ and define 
\begin{align*}
N_{2d}^{\textrm{Weyl}}(X):=\sum_{G \leq S_d}N_{2d}^{\textrm{Weyl}}(X,G). 
\end{align*}
Further, let $N_{2d}^{\textrm{cm}}(X)$ count the CM fields $E$ of degree $2d$ with discriminant $|d_E| \leq X$. Assuming a weak form of the upper bound 
in Malle's conjecture which is known to be true in many cases, we build upon an approach of Kl\"uners to prove that 
\begin{align*}
\frac{N_{2d}^{\textrm{Weyl}}(X,G)}{N_{2d}^{\textrm{cm}}(X)} = C(d, G) + O(X^{-\alpha(d,G)}) 
\end{align*}
and 
\begin{align}\label{abstracteq}
\frac{N_{2d}^{\textrm{Weyl}}(X)}{N_{2d}^{\textrm{cm}}(X)} = 1 + O(X^{-\beta(d)}) 
\end{align}
for some explicit positive constants $C(d,G), \alpha(d,G)$, and $\beta(d)$. 

We then apply these distribution results to study the Colmez conjecture. Using 
the recently proved averaged Colmez conjecture, 
we deduce that the Colmez conjecture is true for $G$-Weyl CM fields. Combined with (\ref{abstracteq}), we conclude that the Colmez conjecture is 
true for an asymptotic density of 100\% of CM fields of degree $2d$; in other words, the Colmez conjecture is true for a random CM field. 

\end{abstract}

\maketitle

\section{Introduction and statement of results}

The distribution of number fields with prescribed Galois group has been studied extensively over the last two decades, spurred in part  
by very precise conjectures of Malle \cite{Mal02, Mal04} for the asymptotic growth of the corresponding counting functions. 
In this paper, we will study this distribution problem for CM fields. 

Recall that a CM field $E$ of degree $2d$ is a totally imaginary quadratic
extension of a totally real field $F$ of degree $d$ over $\Q$. Let $E^c$ be the Galois closure of $E$ and  
$S_d$ be the symmetric group. Then the Galois group $\Gal(E^c/\Q)$ embeds as a subgroup of the 
wreath product $C_2 \wr S_d$. In their study of special points on Shimura varieties, 
Chai and Oort  \cite{CO12} defined $E$ to be a \textit{Weyl} CM field if
$\Gal(E^c/\Q) \cong C_2 \wr S_d$. 
The Weyl CM fields are associated to special 
CM points on the moduli space  
of principally polarized abelian varieties of dimension $d$ called Weyl CM points. 

Now, let $G$ be the transitive subgroup of $S_d$ such that $\Gal(F^c/\Q) \cong G$. Then 
$\Gal(E^c/\Q)$ embeds as a subgroup of the wreath product $C_2 \wr G$ (see e.g. Proposition \ref{embed}). This is a refinement of the 
above mentioned embedding into $C_2 \wr S_d$. 
We define $E$ to be a \textit{$G$-Weyl} CM field if $\Gal(E^c / \Q) \cong C_2 \wr G$. In other words, $E^c$ has maximal possible Galois group, 
subject to the restriction that its maximal totally real subfield $F^c$ have Galois group $G$. 
With this terminology, a Weyl CM field in the sense of Chai and Oort is an $S_d$-Weyl CM field.

Assuming a weak form of the upper bound in Malle's conjecture which is known to be true in many cases 
(see Hypothesis \ref{WeakMalle} and Corollary \ref{Gpairs}), 
we will prove an asymptotic formula with a power-saving error term for the number of $G$-Weyl CM fields $E$ of degree $2d$ with discriminant $|d_E| \leq X$. 
This asymptotic formula implies that for any fixed choice of transitive subgroup $G \leq S_d$, 
a positive proportion of CM fields of degree $2d$ are $G$-Weyl, and moreover, that as $G$ ranges over the transitive subgroups of $S_d$, 
these fields collectively comprise an asymptotic density of 100\% of 
CM fields of degree $2d$ (see Theorem \ref{WeylDensity}).

\begin{remark}
In  \cite[p. 5]{Oor12}, Oort suggests that it is likely that ``most CM fields are [$S_d$-]Weyl CM fields''; such a hypothesis is in line
with a general heuristic that `Galois groups like to be as big as possible'. Perhaps surprisingly, our results show that this heuristic doesn't hold
here. For example, the $S_3$-Weyl CM fields comprise approximately $31\%$ of CM fields of degree $6$, while the 
$C_3$-Weyl CM fields comprise approximately $69\%$ of CM fields of degree $6$. Similarly, the $S_4$-Weyl 
CM fields comprise approximately $20\%$ of CM fields of degree $8$, while the $D_4$-Weyl CM 
fields comprise approximately $48\%$ of CM fields of degree $8$. See Table \ref{Table2} for these statistics.

\end{remark}

Our approach to counting $G$-Weyl CM fields is based on work of Kl\"uners \cite{Klu12}, which established asymptotics 
for the counting function of number fields with Galois group $C_2 \wr G$, without signature conditions.
We will adapt Kl\"uners' work to handle the signature conditions needed to count CM fields. We also give 
power-saving error terms which incorporate recent progress on
non-trivial bounds for 2-torsion in class groups of number fields \cite{BSTTTZ17} and subconvexity bounds for ray class $L$--functions 
of totally real fields \cite{ELMV11}, and determine the weakest form of the upper bound in Malle's conjecture needed for our results.

We next discuss the connection between the distribution of $G$-Weyl CM fields 
and the Colmez conjecture \cite{Col93}, which relates the Faltings height of a CM abelian variety to logarithmic derivatives 
of Artin $L$--functions at $s=0$. In fact, this paper was motivated in part by our effort to answer the following:   

\vspace{0.10in}

\textbf{Question}. \textit{Is the Colmez conjecture true for a random CM field?}

\vspace{0.10in}

To address this problem, the first two authors \cite[Section 1.2]{BSM16} 
developed a plan to study the Colmez conjecture from an arithmetic statistical 
point of view. Using the averaged Colmez conjecture, which was proved independently by 
Andreatta, Goren, Howard, and Madapusi Pera \cite{AGHM15}, and Yuan and Zhang \cite{YZ15}, the first two authors \cite[Theorem 1.4]{BSM16} deduced 
that the Colmez conjecture is true for $S_d$-Weyl CM fields. Then, they \cite[Theorem 1.9]{BSM16}
applied work of Cohen, Diaz y Diaz, and Olivier \cite{CDO02} 
to conclude that 100\% of quartic CM fields are $S_2$-Weyl, and consequently, satisfy the Colmez conjecture. 

Due to the well known difficulties which arise when counting number fields with Galois group $S_d$, 
this line of attack seemed limited initially to CM fields of small degree. 
Here we overcome these difficulties by first using the averaged Colmez conjecture and the Galois theory of CM fields 
to deduce that the Colmez conjecture is true for any $G$-Weyl CM field. Then, combined with our distribution results, 
we will conclude (conditional on Hypothesis \ref{WeakMalle}) that the 
Colmez conjecture is true for 100\% of CM fields of degree $2d$; in other words, the Colmez conjecture is  
true for a random CM field (see Theorem \ref{Colmez2}). Moreover, given the pairs 
$(d,G)$ for which Hypothesis \ref{WeakMalle} is known unconditionally, we will produce 
infinitely many density-one families of non-abelian CM fields which satisfy the Colmez conjecture
(see the results of Section \ref{applications}).

\subsection{The distribution of $G$-Weyl CM fields}\label{subsec:dist}

In order to state our distribution results, we first define some of the counting functions that will be used throughout the paper (all number 
fields are counted up to $\Q$-isomorphism). 

\vspace{0.10in}

\begin{itemize}

\item Let $N_d(X, G)$ count all number fields $K$ of degree $d$ and discriminant 
$|d_K| \leq X$ with $\Gal(K^c/\Q) \cong G$.
\vspace{0.10in}

\item Let $N_{2d}^{\textrm{cm}}(X, G)$ count all CM fields $E$ of degree $2d$ and discriminant $|d_E | \leq X$ 
which have a maximal totally real subfield $F$ with $\Gal(F^c/\Q) \cong G$. 
\vspace{0.10in}

\item Let $N_{2d}^{\textrm{Weyl}}(X, G)$ count all CM fields $E$ of degree $2d$ and discriminant 
$|d_E| \leq X$ which are $G$-Weyl.
\vspace{0.10in}

\item Let $$N_{2d}^{\textrm{Weyl}}(X):=\sum_{G \leq S_d}N_{2d}^{\textrm{Weyl}}(X,G).$$ 

\item Let $N_{2d}^{\textrm{cm}}(X)$ count all CM fields $E$ of degree $2d$ and discriminant $|d_E| \leq X$. 
\end{itemize}

\vspace{0.10in}

In \cite{Mal02}, Malle gave conjectural bounds for $N_d(X,G)$. If $g \in S_d$, the \textit{index} of $g$ is 
defined by 
\begin{align*}
\textrm{ind}(g):= d - \#(\{1, \ldots, d\}/\langle g \rangle).
\end{align*}
Let 
\begin{align*}
\textrm{ind}(G):=\min\{\textrm{ind}(g):~1 \neq g \in G\}
\end{align*}
and define the constant $0 < a(G):=\textrm{ind}(G)^{-1} \leq 1$. 
Malle conjectured that for any $\epsilon > 0$, there exist positive constants $c_1(G), c_2(G, \epsilon)$ such that 
\begin{align}\label{wmc}
 c_1(G) X^{a(G)}  \leq N_d(X, G) \leq c_2(G, \epsilon) X^{a(G)+\epsilon}.
\end{align} 
Malle \cite{Mal04} later refined this and gave a precise conjectural asymptotic formula for $N_d(X,G)$ as $X \rightarrow \infty$.

For our purposes, we only need an upper bound for $N_d(X,G)$ with 
an exponent which is \textit{much} weaker than what is predicted by (\ref{wmc}). 
This exponent will depend on bounds for $2$-torsion in class groups. 

For a number field $K$, let 
$\textrm{Cl}(K)[2]$ be the $2$-torsion subgroup of the ideal class group. Let $\delta_d \geq 0$ be a variable 
such that 
\begin{align}\label{torsion}
|\textrm{Cl}(K)[2]| \ll_{\epsilon, d} |d_K|^{\delta_d + \epsilon}
\end{align}
for all number fields $K$ of degree $d$.
By the Brauer-Siegel theorem, the bound (\ref{torsion}) holds with $\delta_d = 1/2$. 
Any bound (\ref{torsion}) with $ 0 < \delta_d < 1/2$ is called a non-trivial bound, and $\delta_d =0$ is the conjectured optimal 
bound. 

If $d=2$, then it is a classical result that (\ref{torsion}) holds with $\delta_2=0$.
The first non-trivial bounds in (\ref{torsion}) for $d \geq 3$  
were recently proved by Bhargava, Shankar, Taniguchi, Thorne, Tsimerman, and Zhang \cite{BSTTTZ17}. In particular, 
they proved that if $d=3,4$, then (\ref{torsion}) holds with $\delta_d=0.2784$, and if $d \geq 5$, 
then (\ref{torsion}) holds with $\delta_d=1/2-1/2d$.

With the variable $\delta_d$ as in (\ref{torsion}), we state 
the following weak form of the upper bound in Malle's conjecture (\ref{wmc}). 

\begin{hypothesis}\label{WeakMalle} For a fixed pair $(d,G)$ and $0 \leq \delta_d \leq 1/2$ satisfying (\ref{torsion}), we have
\begin{equation}\label{eq:malle}
N_d(X, G) \ll X^{M(G)}
\end{equation}
for some $M(G) > 0$ such that 
\begin{align*}
\delta_d + M(G) < 2.
\end{align*}
\end{hypothesis}

Our first result gives an asymptotic formula with a power-saving error term for the density of those CM fields 
counted by $N_{2d}^{\textrm{cm}}(X,G)$ which are $G$-Weyl.

\begin{theorem}\label{main_thm}
Assume that Hypothesis \ref{WeakMalle} is true for $(d, G)$. Then 
\begin{equation}\label{density1}
\frac{N_{2d}^{\mathrm{Weyl}}(X, G)}{N_{2d}^{\mathrm{cm}}(X, G)} = 1 + O_{d,G, \epsilon}(X^{-C_1(\delta_d, M(G)) + \epsilon}),
\end{equation}
where 
\begin{equation}\label{eq:def_C1}
C_1(\delta_d,M(G)):= 
\begin{cases}
1/2, & \textrm{if $\delta_d + M(G) \leq 1$}\\
\displaystyle 1- \frac{\delta_d + M(G)}{2}, &  \textrm{if $1 < \delta_d + M(G) < 2$}.
\end{cases}
\end{equation}
%In particular, 
%we have
%\begin{align}\label{density1}
%\lim_{X \rightarrow \infty} \frac{N_{2d}^{\mathrm{Weyl}}(X, G)}{N_{2d}^{\mathrm{cm}}(X, G)}=1.
%\end{align}
\end{theorem}

The Hypothesis \ref{WeakMalle} is known in many cases due to work of the following 
authors on the Malle conjectures: \cite{DH71, Mak85, CDO02, KM04, Bha05, KY05, EV06, BW08, Bha10, CT17, Wan17, Klu12}.
For convenience, we have summarized these results 
in Table \ref{Table1} of Section \ref{known}.

Given these known cases of Hypothesis \ref{WeakMalle}, we get the following unconditional results.

\begin{corollary}\label{Gpairs} 
The asymptotic formula \eqref{density1} holds unconditionally for the following pairs $(d,G)$: 
\begin{itemize}
\item Any $(d, G)$ with $G$ abelian.
\item Any $(d,G)$ with $d=\ell$ prime and $G=D_{\ell}$ dihedral.
\item Any $(d,G)$ with $G$ a $p$-group.
\item Any $(d,G)$ with $d \geq 5$ and $|G|=d$.
%\footnote{The condition $|G|=d$ 
%is equivalent to all number fields counted by $N_{d}(X,G)$ being Galois over $\Q$. This case follows from Ellenberg and Venkatesh 
%\cite[Proposition 1.3]{EV06}.}
\item  Any $(d,G)$ with $d\leq 5$.
\item Any $(d,G)$ with $d=3|A|$ and $G=S_3 \times A$ with $A$ abelian.
\item Any $(d,G)$ with $d=4|A|$ and $G=S_4 \times A$ with $A$ abelian.
\item Any $(2d, C_2 \wr G)$, when $(d, G)$ is on this list.
\end{itemize}
\end{corollary}
\vspace{0.025in}

\begin{remark} In the 4th bullet of Corollary \ref{Gpairs}, the condition $|G|=d$ is equivalent 
to all number fields counted by $N_{d}(X,G)$ being Galois over $\Q$. This case follows from Ellenberg and Venkatesh 
\cite[Proposition 1.3]{EV06}.
\end{remark}

\begin{remark} The Malle conjecture can be formulated more generally for degree $d$ extensions $L/K$ of any 
global field $K$ with $\mathcal{N}_{K/\Q}(\frak{D}_{L/K}) < X$ and Galois group $\textrm{Gal}(L/K) \cong G$ for a transitive subgroup 
$G \leq S_d$ (here $\frak{D}_{L/K}$ is the relative discriminant). In this setting, Ellenberg, Tran, and Westerland \cite{ETW17} 
recently proved the upper bound in Malle's Conjecture when $K=\mathbb{F}_q(t)$ is the rational function field.
\end{remark}

\begin{example} If $(d,G)=(5,S_5)$, then   
Hypothesis \ref{WeakMalle} is true for the pair $(\delta_5, M(S_5))=(2/5,1)$. Since 
$C_1(2/5,1)=3/10$, we have   
\begin{align*}
\frac{N_{10}^{\mathrm{Weyl}}(X, S_5)}{N_{10}^{\mathrm{cm}}(X, S_5)} = 1 + O_{\epsilon}(X^{-\frac{3}{10} + \epsilon}).
\end{align*}
\end{example}

\begin{remark} Assuming a sufficiently strong value for the exponent 
$\delta_d$ appearing in the $2$-torsion bound (\ref{torsion}),  the asymptotic formula (\ref{density1}) is known for 
some additional pairs $(d,G)$; see Table \ref{Table4}.
\end{remark}

Our next goal is to give an asymptotic formula with a power-saving error term for the density of those CM 
fields counted by $N_{2d}^{\textrm{cm}}(X)$ which are $G$-Weyl. 
This will involve the subconvexity problem for a certain family of Hecke $L$--functions for totally real fields. 

Let $F$ be a totally real field of degree $d$. Let    
$\mathfrak{c}$ be an integral ideal of $F$ dividing $2$, and let $\mathfrak{c}_{\infty} \subset \mathfrak{m}_{\infty}$ 
be a subset of the set $\mathfrak{m}_{\infty}$ of real places of $F$. Suppose that $\chi$ is a primitive character of the 
ray class group $\Cl_{\mfc^2 \mfc_\infty}(F)$ modulo $\mfc^2 \mfc_\infty$. The $L$--function of $\chi$ is defined by 
\begin{align*} 
L_F(\chi,s):=\prod_{\frak{p}}\left(1-\chi(\frak{p})\mathcal{N}_{F/\Q}(\frak{p})^{-s}\right)^{-1}, \quad \textrm{Re}(s) > 1.
\end{align*}
The completed $L$--function is defined by (see e.g. \cite[p. 129]{IK04})
\begin{align*}
\Lambda_F(\chi,s):=q(\chi)^{s/2}\gamma(\chi,s)L_F(\chi,s),
\end{align*}
where $q(\chi):=d_F\mathcal{N}_{F/\Q}(\frak{c}^2)$ and 
\begin{align*}
\gamma(\chi,s):=\pi^{-ds/2}\Gamma\left(\frac{s}{2}\right)^{d-|\frak{c}_{\infty}|}\Gamma\left(\frac{s+1}{2}\right)^{|\frak{c}_{\infty}|}.
\end{align*}
The completed $L$--function satisfies the functional equation
\begin{align*}
\Lambda_F(\chi,s)=\varepsilon(\chi)\Lambda_F(\overline{\chi},1-s),
\end{align*}
where the root number $\varepsilon(\chi)$ is a complex number of modulus 1 which can be written explicitly as a normalized 
Gauss sum for $\chi$. Given this data, we calculate the analytic conductor of $L_F(\chi,s)$ as (a slightly weaker version of \cite[eq. (5.7)]{IK04})
\begin{align*}
\mathfrak{q}(F,\chi,s)=d_F\mathcal{N}_{F/\Q}(\mfc^2)(|s| + 4)^d.
\end{align*}

Let $\delta^{\prime} \geq 0$ be a variable such that 
\begin{align}\label{subconvex}
\left(\frac{s-1}{s+1}\right)^{a(\chi)}L_F(\chi,s) \ll_{\epsilon, d} \mathfrak{q}(F,\chi,s)^{\delta^{\prime}(1-\sigma)+\epsilon}, \quad 
1/2 \leq \sigma:=\textrm{Re}(s) \leq 1 + \epsilon,
\end{align}
where $a(\chi)=1$ if $\chi$ is the trivial character 
and $a(\chi)=0$ if $\chi$ is non-trivial.
The bound (\ref{subconvex}) holds when $\delta^{\prime} = 1/2$ (the convexity bound). Any bound 
(\ref{subconvex}) with $0 < \delta^{\prime} < 1/2$ is called a subconvexity bound, and $\delta^{\prime}=0$ is the Lindel\"of hypothesis. 
For more details concerning these facts, see \cite[Chapter 5]{IK04}.

\begin{remark} A subconvexity bound of the form (\ref{subconvex}) is known, for example, 
if $F$ is either abelian or cubic (see e.g. the summary of results in \cite[Appendix A]{ELMV11}).   
\end{remark}

Now, let $D_{G}^{\textrm{cm}}(s)$ be the Dirichlet series which enumerates all fields counted 
by $N_{2d}^{\text{cm}}(X, G)$ (see (\ref{eqn:def_phicm})). In Theorem \ref{thm:kl5}, we will prove that if 
Hypothesis \ref{WeakMalle} is true for $(d,G)$ and $0 \leq \delta^{\prime} \leq 1/2$ satisfies (\ref{subconvex}), 
then $D_G^{\textrm{cm}}(s)$ has a meromorphic continuation to a half-plane $\textrm{Re}(s) > \alpha$ 
for some $\alpha < 1$ (depending on $\delta_d, M(G)$ and $\delta^{\prime}$) 
with only a single (simple) pole at $s=1$. Moreover, the residue of $D_G^{\textrm{cm}}(s)$ at $s=1$ is given by the convergent series 
\begin{align}\label{residue1}
r_d(G):=\sum_{F \in \mathcal{F}_G^{+}}\frac{\textrm{Res}_{s=1}\zeta_F(s)}{2^dd_F^2\zeta_F(2)} > 0,
\end{align}
where 
\begin{align*}
\mathcal{F}_G^{+} :=\{F/\Q:~ \textrm{$F$ totally real of degree $d$}, ~\textrm{Gal}(F^c/\Q) \cong G \}.
\end{align*}

Using properties of the Dirichlet series $D_{G}^{\textrm{cm}}(s)$ and an upper bound for the 
number of CM fields counted by $N_{2d}^{\cm}(X)$ which are 
\textit{not} $G$-Weyl for any transitive subgroup $G \leq S_d$, we will prove the following 
asymptotic formulas with power-saving error terms.

\begin{theorem}\label{WeylDensity}
Assume that Hypothesis \ref{WeakMalle} is true for every pair $(d, G)$ where $G$ ranges over all  
transitive subgroups $G \leq  S_d$. Moreover, assume that $0 \leq \delta^{\prime} \leq 1/2$ satisfies (\ref{subconvex}). Then for 
any such $G_0 \leq S_d$, we have 
\begin{align}\label{full1}  
\frac{N_{2d}^{\Weyl}(X, G_0)}{N_{2d}^{\cm}(X)} = \frac{r_d(G_0)}{\sum_{G \leq S_d}r_d(G)}
+ O_{d,G_0,\epsilon}(X^{-C_2(\delta_d, M(G_0), \delta^{\prime}) + \epsilon}), 
\end{align}
and 
\begin{align}\label{full2}
\frac{N_{2d}^{\mathrm{Weyl}}(X)}{N_{2d}^{\mathrm{cm}}(X)} = 1 
+ O_{d,\epsilon}(X^{-C_3(\delta_d, \delta^{\prime}) + \epsilon}),
\end{align}
where $C_2(\delta_d, M(G), \delta^{\prime}) > 0$ and $C_3(\delta_d, \delta^{\prime}) > 0$ are explicit constants defined in \eqref{eq:def_C2} 
and \eqref{eq:def_C3}, respectively.
\end{theorem}

In Table \ref{Table2} we give numerical computations for the residue $r_d(G)$, and hence for the relative density
of $G$-Weyl CM fields, for each transitive $G \leq S_d$ with $d \leq 5$. We computed these by
summing the series of \eqref{residue1} over the first $n$ fields $F \in \mathcal{F}_G^+$, for $n$ listed in the table.
The basic field data was downloaded from the website \url{lmfdb.org} \cite{LMFDB}, and the remaining computations,
including the $L$-function computations in \eqref{residue1}, were handled with PARI/GP \cite{PARI}. The (short) PARI/GP source code
with which we put these computations together may be downloaded at the third author's website\footnote{\url{http://people.math.sc.edu/thornef}}. 

From \eqref{residue1} we see that the residues are (very) approximately given by 
$2^{-d} \sum_F d_F^{-2}.$
Assuming Malle's conjecture \eqref{wmc}, the series converge relatively rapidly; and indeed it is known
that $N_d(X, G) \ll X$ for all $(d, G)$ listed in the table. With some effort, it should be possible to explicitly bound the error
in our residue computations below; numerics suggest that these values are likely to be accurate within approximately
 $\pm 1$ in the least significant digit listed.

We observe: each totally real $F$ contributes a positive proportion to its respective residue, with those
of smallest discriminant making the largest contribution; also, the residues are decreasing with $d$ -- a pattern which should persist, in light of
lower bounds on $d_F$ which are exponential in $d$ \cite{odlyzko}.

%\small
\begin{table}[H]
\caption{Values of $r_d(G)$ for $d \leq 5$} \label{Table2} 
\vspace{0.10in}
{\tabulinesep=1.0mm
\begin{tabu}{ |c|c|c|c|c|c| }
\hline
$d$ & $G$ & Number of fields & Minimal discriminant & Residue & Proportion in \eqref{full1} \\ 
\hline
2 & $C_2$ & 100,000 & $5$ & $0.009856$ & - \\ \hline
3 & & 25,000 & $49$ & $3.30 \times 10^{-5}$ & - \\
& $C_3$ & 107 & $49$ & $2.29 \times 10^{-5}$ & $0.69$\\
& $S_3$ & 24,893 & $148$ & $1.01 \times 10^{-5}$ & $0.31$ \\ \hline
4 & & 25,000 & $725$ & $1.24 \times 10^{-7}$ & - \\
& $C_4$ & 75 & $1125$ & $2.41 \times 10^{-8}$ & $0.19$ \\
& $V_4$ & 289 & $1600$ & $1.56 \times 10^{-8}$ & $0.13$ \\
& $D_4$ & 8147 & $725$ & $5.9 \times 10^{-8}$ & $0.48$ \\
& $A_4$ & 45 & $26569$ & $9.3 \times 10^{-11}$ & $0.0008$ \\
& $S_4$ & 16,444 & $1957$ & $2.5 \times 10^{-8}$ & $0.20$ \\ \hline
5 & & 25,000 & $14641$ & $1.05 \times 10^{-10}$ & - \\
& $C_5$ & 5 & $14641$ & $3.08 \times 10^{-11}$ & $0.29$ \\
& $D_5$ & 28 & $160801$ & $4.24 \times 10^{-13}$ & $0.003$ \\
& $F_5$ & 15 & $2382032$ & $9 \times 10^{-15}$ & $0.00009$ \\
& $A_5$ & 21 & $3104644$ & $5 \times 10^{-15}$ & $0.00005$ \\
& $S_5$ & 24,931 & $24217$ & $7.4 \times 10^{-11}$  & $0.70$ \\ \hline
\end{tabu}}
\end{table}
%\normalsize

\begin{example} Let $(d,G)$ be any pair with $d=5$. Then Hypothesis \ref{WeakMalle} is true 
for the pair $(\delta_5,M(G))=(2/5,1)$. If we take $\delta^{\prime}=1/2$, then we have
\small
\begin{align*}
C_1(\delta_5, M(G)) \geq \frac{1}{4}, \
\alpha(\delta_5, M(G), \delta') \leq \frac{19}{25}, \
\beta(\delta_5, \delta') = \frac{17}{20}, \
C_2(\delta_d, M(G), \delta') = \frac{3}{20}, \
C_3(\delta_d, \delta') = \frac{3}{20},
\end{align*}
\normalsize
these constants being defined in 
\eqref{eq:def_C1},
\eqref{alphadef},
\eqref{betadef},
\eqref{eq:def_C2}, and
\eqref{eq:def_C3} respectively. We conclude that
\begin{align*}
\frac{N_{10}^{\mathrm{Weyl}}(X)}{N_{10}^{\mathrm{cm}}(X)} = 1 
+ O_{\epsilon}(X^{-\frac{3}{20} + \epsilon}).
\end{align*}

Moreover, if we assume the conjectured optimal bound in (\ref{torsion}) and the 
Lindel\"of hypothesis in (\ref{subconvex}) (so that 
Hypothesis \ref{WeakMalle} is true for the pair $(\delta_5,M(G))=(0,1)$ and $\delta^{\prime}=0$), then  
\begin{align*}
\frac{N_{10}^{\mathrm{Weyl}}(X)}{N_{10}^{\mathrm{cm}}(X)} = 1 
+ O_{\epsilon}(X^{-\frac{1}{2} + \epsilon}).
\end{align*}
\end{example}

Since Hypothesis \ref{WeakMalle} is known for every pair $(d,G)$ with $d \leq 5$ (see Corollary \ref{Gpairs}), we get the following 
unconditional result.

\begin{corollary}\label{densityall} 
If $d \leq 5$, then $(\ref{full1})$ and $(\ref{full2})$ hold unconditionally. 
In particular, if $d \leq 5$, then the set  
\begin{align*}
%S_{\mathrm{Weyl}}(d):=
\bigcup_{G \leq S_d}\{\textrm{$G$-Weyl CM fields of degree $2d$}\}
\end{align*}
of all Weyl CM fields of degree $2d$ comprises 100\% of all CM fields of degree $2d$. 
\end{corollary}

\subsection{Application to the Colmez conjecture}\label{applications} 

As discussed, we will use the 
averaged Colmez conjecture \cite{AGHM15, YZ15} and the Galois theory of CM fields to deduce the following:

\begin{theorem}\label{ColmezWeyl} If $E$ is a $G$-Weyl CM field, then the Colmez conjecture is true for $E$. 
In particular, if $X$ is an abelian variety of dimension $d$ with complex multiplication by a $G$-Weyl CM field $E$ of degree $2d$ 
with maximal totally real subfield $F$, then the Faltings height of $X$ is given by 
\begin{align}\label{nice}
h_{\mathrm{Fal}}(X)= -\frac{1}{2}\frac{L^{\prime}(\chi_{E/F},0)}{L(\chi_{E/F},0)} -\frac{1}{4}\log\left(\frac{|d_E|}{d_F}\right) - \frac{d}{2}\log(2\pi),
\end{align} 
where $L(\chi_{E/F},s)$ is the $L$--function of the Hecke character $\chi_{E/F}$ associated to the quadratic extension $E/F$.
\end{theorem}

\begin{remark}\label{keyremark} If $E$ is a $G$-Weyl CM field of degree $2d \geq 4$, then $E/\Q$ is \textit{non-abelian} (see Proposition \ref{non-abelian}).
\end{remark}

The following results give infinitely many density-one families of non-abelian CM fields which satisfy
the Colmez conjecture.

First, we have the following result, which is an immediate consequence of Theorems \ref{ColmezWeyl} and \ref{main_thm}.

\begin{theorem}\label{Colmez1} Assume that Hypothesis \ref{WeakMalle} is true for the pair $(d, G)$. Then the Colmez conjecture is true 
for 100\% of CM fields $E$ of degree $2d$ which have a maximal totally real subfield $F$ with Galois group 
$\mathrm{Gal}(F^c/\Q) \cong G$. 
%when ordered by discriminant.
\end{theorem}

Next, by combining Theorem \ref{Colmez1} with Corollary \ref{Gpairs}, we get the following unconditional result.

\begin{corollary}\label{Colmez1Cor} If $(d,G)$ is any of the pairs in Corollary \ref{Gpairs}, then 
the Colmez conjecture is true for 100\% of CM fields $E$ of degree $2d$ which have a 
maximal totally real subfield $F$ with Galois group $\mathrm{Gal}(F^c/\Q) \cong G$. 
\end{corollary}

Similarly, the following is an immediate consequence of Theorems \ref{ColmezWeyl} and \ref{WeylDensity}.

\begin{theorem}\label{Colmez2} 
Assume that Hypothesis \ref{WeakMalle} is true for every pair $(d, G)$ where $G$ ranges over all 
transitive subgroups $G \leq  S_d$. Then the Colmez conjecture is true for 100\% of CM fields of degree $2d$. 
\end{theorem}

Finally, since Hypothesis \ref{WeakMalle} holds for every pair $(d,G)$ with $d \leq 5$ (as observed previously), we get the following 
unconditional result.

\begin{corollary}\label{ColmezAll} 
If $d \leq 5$, then the Colmez conjecture is true for 100\% of CM fields of degree $2d$. 
\end{corollary}

\begin{remark} The Colmez conjecture is now known to be true 
for quartic CM fields, sextic CM fields, and many degree 10 CM fields (see e.g. \cite{YY16}).
\end{remark}

We conclude by briefly summarizing some results on the Colmez conjecture. 

Colmez \cite{Col93} proved his conjecture for abelian CM fields (up to an error term which was eliminated by Obus \cite{Obu13}). 

Tonghai Yang \cite{Yan10a, Yan10b, Yan13} proved the Colmez conjecture for a large class of quartic CM fields, including the first non-abelian cases. 

The averaged Colmez conjecture \cite{AGHM15, YZ15} made it possible to deduce many new cases of the Colmez conjecture. 
For example, the first two authors \cite{BSM16} proved that if 
$F$ is any totally real number field of degree $d \geq 3$, 
then there are infinitely many effectively constructible, positive density sets of CM extensions 
$E/F$ such that $E/\Q$ is non-abelian and the Colmez conjecture is true for $E$. 
Yang and Yin \cite{YY16} proved that if $E$ is a CM field of the form $E=FK$ 
where $K=\Q(\sqrt{-D})$ is an imaginary quadratic field, and  
$\textrm{Gal}(F^c/\Q)$ is isomorphic to either $S_d$ or $A_d$, then the Colmez conjecture is true for $E$. This follows from a more refined result
of these authors \cite{YY16} which shows that the Colmez conjecture is true for CM types $\Phi$ of $E=FK$ of signature $(d-1,1)$. 
Parenti \cite{P18} proved that if $E$ is a CM field of the form $E=FK$, 
and $\textrm{Gal}(F^c/\Q) \cong PSL_2(\mathbb{F}_q)$, then the Colmez conjecture is true for $E$.

\section{Proof of Theorems \ref{main_thm} and \ref{WeylDensity}}

In this section we prove Theorems \ref{main_thm} and \ref{WeylDensity}, following closely Kl\"uners's Dirichlet series approach  \cite{Klu12}.  

Let $N_{2d}^{\textrm{cm}}(X, G)$ count all CM fields $E$ of degree $2d$ and discriminant $|d_E | \leq X$ 
which have a maximal totally real subfield $F$ with $\Gal(F^c/\Q) \cong G$. 

Let $N_{2d}^{\neg \Weyl}(X, G)$ count the subset of all CM fields counted by $N_{2d}^{\textrm{cm}}(X, G)$ 
which are \textit{not} of $G$-Weyl type. 

We first establish asymptotics for $N_{2d}^{\cm}(X, G)$, and then give upper bounds for $N_{2d}^{\neg  \Weyl}(X, G)$.

\subsection{Asymptotics for  $N_{2d}^{\cm}(X, G)$} We begin by establishing an asymptotic formula with a power-saving 
error term for $N_{2d}^{\cm}(X, G)$. The key is a theorem of Cohen, Diaz y Diaz, and Olivier \cite{CDO02} which expresses the Dirichlet series enumerating 
all quadratic extensions of a number field as a linear combination of Hecke $L$--functions. 
We will use a version of their result which incorporates signature conditions.

Fix a totally real field $F$ of degree $d$ and define the Dirichlet series
\begin{equation*}
D_{F, C_2}^-(s) := \sum_{[E : F] = 2} \frac{1}{\calN_{F/\Q}(\frak{D}_{E/F})^s}=\sum_{n=1}^{\infty}\frac{a^{-}(n)}{n^{s}},  \quad \textrm{Re}(s) > 1 
%= \sum_{n = 1}^{\infty}
%\frac{a^-(n)}{n^s},
\end{equation*}
where the sum is over all totally imaginary quadratic extensions $E/F$, $\frak{D}_{E/F}$ is the relative discriminant, and 
\begin{align*}
a^{-}(n):=\# \{\textrm{$E/F$ totally imaginary quadratic}, ~ \mathcal{N}_{F/\Q}(\frak{D}_{E/F})=n\}.
\end{align*}

The following is a special case of 
\cite[Theorem 3.11]{CDO02}, applied to the totally real field $F$ (which has signature $(d,0)$) and with 
$\mathfrak{m}_{\infty}$ equal to the set of real places of $F$ 
(which are all ramified in the imaginary quadratic extension $E/F$). 

\begin{theorem}[\cite{CDO02}]
For $\mathrm{Re}(s) > 1$ we have
\begin{equation*}
D_{F, C_2}^-(s) = \frac{1}{\zeta_F(2s)}
\sum_{\frak{c}_{\infty} \subset \frak{m}_{\infty}}\sum_{\mfc \mid 2} \frac{ (-1)^{|\mfc_\infty|}}{2^{|\mfc_\infty|}}
\calN_{F/\Q}(2/\mfc)^{1 - 2s} \sum_{\chi \in Q(\Cl_{\mfc^2 \mfc_\infty}(F))} L_F(\chi, s),
\end{equation*}
where $\mfc$ runs over all integral ideals of $F$ dividing $2$, $\mfc_\infty$ runs over
all subsets of the set of real places $\frak{m}_{\infty}$ of $F$, $\chi$ runs over all quadratic characters 
$Q(\Cl_{\mfc^2 \mfc_\infty}(F))$ of the ray class group $\Cl_{\mfc^2 \mfc_\infty}(F)$ modulo $\mfc^2 \mfc_\infty$, and $L_F(\chi,s)$ 
is the $L$--function of $\chi$.
\end{theorem}

The following result establishes some important analytic properties of the Dirichlet series $D_{F, C_2}^-(s)$ (this is 
analogous to \cite[Theorem 5]{Klu12}).

\begin{theorem}\label{thm:kl5} Assume that $0 \leq \delta_d, \delta^{\prime} \leq 1/2$ satisfy 
(\ref{torsion}) and (\ref{subconvex}), respectively. Then the Dirichlet series $D_{F, C_2}^-(s)$ 
has a meromorphic continuation to $\mathrm{Re}(s) > 1/2$ with only a single (simple) pole at $s = 1$ with residue
\[
R_d(F) := \frac{\Res_{s = 1} \zeta_F(s)}{2^{d}\zeta_F(2)} > 0.
\]
Moreover, the function 
\begin{align*}
g_F(s) := D_{F, C_2}^-(s) - \frac{R_d(F)}{s - 1}
\end{align*}
is analytic for $\sigma:=\mathrm{Re}(s) > 1/2$ and satisfies the bound
\begin{equation}\label{eqn:convexity}
g_F(\sigma + it) \ll_{\epsilon, d} d_F^{\delta^{\prime}(1 - \sigma) + \delta_d + \epsilon} 
\frac{(1 + |t|)^{d\delta^{\prime}(1 - \sigma) + \epsilon}}{(\sigma - 1/2)^{d}}, \quad 1/2 < \sigma \leq 1 + \epsilon.
\end{equation}
\end{theorem}

\begin{proof} 
Let $\chi_{0, \frak{c}} \in Q(\Cl_{\mfc^2 \mfc_\infty}(F))$ be the trivial character and write 
\begin{align}\label{decompAB}
D_{F, C_2}^-(s)=A(s) + B(s),
\end{align}
where 
\begin{align*}
A(s)&:= \frac{1}{\zeta_F(2s)}\sum_{\frak{c}_{\infty} \subset \frak{m}_{\infty}}\frac{(-1)^{|\mfc_\infty|}}{2^{|\mfc_\infty|}}
\sum_{\mfc \mid 2} \calN_{F/\Q}(2/\mfc)^{1 - 2s}L(\chi_{0,\frak{c}},s) 
\end{align*}
and
\begin{align*}
B(s)&:=\frac{1}{\zeta_F(2s)}\sum_{\frak{c}_{\infty} \subset \frak{m}_{\infty}}\sum_{\mfc \mid 2} \frac{ (-1)^{|\mfc_\infty|}}{2^{|\mfc_\infty|}}
\calN_{F/\Q}(2/\mfc)^{1 - 2s} \sum_{\substack{\chi \in Q(\Cl_{\mfc^2 \mfc_\infty}(F)) \\ \chi \neq \chi_{0,\frak{c}}}}L_F(\chi,s).
\end{align*}
The $L$--function  
\begin{align*}
L_F(\chi_{0, \frak{c}},s)=\zeta_F(s)\prod_{\frak{p}|\frak{c}}\left(1-\mathcal{N}_{F/\Q}(\frak{p})^{-s}\right)
\end{align*} 
extends to a meromorphic function on $\C$ with only a single (simple) pole at $s=1$, and 
if $\chi \neq \chi_{0,\frak{c}}$ the $L$--function $L_F(\chi,s)$ extends to an analytic function on $\C$. Hence, 
by (\ref{decompAB}) the Dirichlet series $D_{F, C_2}^-(s)$ extends to a meromorphic function on $\sigma > 1/2$ 
with only a single (simple) pole at $s=1$ with residue 
\begin{align*}
R_d(F)=\textrm{Res}_{s=1}A(s)=S \cdot \textrm{Res}_{s=1}\zeta_F(s),
\end{align*}
where 
\begin{align*}
S:=\frac{1}{\zeta_F(2)}\sum_{\frak{c}_{\infty} \subset \frak{m}_{\infty}}\frac{(-1)^{|\mfc_\infty|}}{2^{|\mfc_\infty|}}
\sum_{\mfc \mid 2} \calN_{F/\Q}(2/\mfc)^{-1}\prod_{\frak{p}|\frak{c}}\left(1-\mathcal{N}_{F/\Q}(\frak{p})^{-1}\right).
\end{align*}
As in \cite[Sections 3.3 and 3.4]{CDO02}, we may compute that
\begin{align*}
\sum_{\mfc \mid 2} \calN_{F/\Q}(2/\mfc)^{-1}\prod_{\frak{p}|\frak{c}}\left(1-\mathcal{N}_{F/\Q}(\frak{p})^{-1}\right)=1
\end{align*}
and 
\begin{align*}
\sum_{\frak{c}_{\infty} \subset \frak{m}_{\infty}}\frac{(-1)^{|\mfc_\infty|}}{2^{|\mfc_\infty|}}=\frac{1}{2^{|\frak{m}_{\infty}|}}=\frac{1}{2^{d}}.
\end{align*}
Therefore, we have $S=2^{-d}$ and  
\begin{align*}
R_d(F)=\frac{\textrm{Res}_{s=1}\zeta_F(s)}{2^d\zeta_F(2)}.
\end{align*}

By the preceding facts, the function $g_F(s):=D_{F, C_2}^-(s)-R_d(F)/(s-1)$ is analytic for $\sigma > 1/2$. 
Hence, it remains to establish the bound (\ref{eqn:convexity}).

By (\ref{subconvex}), we have the bound 
\begin{align*}
(s-1)\zeta_F(s) & \ll_{\epsilon, d} \mathfrak{q}(F,\chi_0,s)^{\delta^{\prime}(1-\sigma)+\epsilon}(|s|+1) \\ 
& \ll_{\epsilon, d} (d_F(|s| + 4)^d)^{\delta^{\prime}(1-\sigma)+\epsilon}(|s|+1), \quad 1/2 < \sigma \leq 1 + \epsilon,
\end{align*}
and we also have the bound 
\begin{align*}
\frac{1}{\zeta_F(2s)} \ll \frac{1}{(\sigma - 1/2)^{d}}, \quad 1/2 < \sigma \leq 1 + \epsilon
\end{align*}
(the implied constant is uniform in $F$).
These bounds yield the estimate 
\begin{align*}
%\label{Ae1}
(s-1)A(s) - R_d(F) \ll_{\epsilon, d} (\sigma - 1/2)^{-d}  
(d_F(|s| + 4)^d)^{\delta^{\prime}(1-\sigma)+\epsilon}(|s| + 1), \quad 1/2 < \sigma \leq 1+ \epsilon,
\end{align*}
and thus with $f(s):=A(s)-R_d(F)/(s-1)$ the estimate

\begin{align*}
f(\sigma + it) \ll_{\epsilon, d} (\sigma - 1/2)^{-d} d_F^{\delta^{\prime}(1-\sigma)+\epsilon}(1 + |t|)^{d(\delta^{\prime}(1-\sigma)+\epsilon)}, 
\quad 1/2 < \sigma \leq 1 + \epsilon.
\end{align*}

Next observe that if the bound (\ref{subconvex}) holds for some $0 \leq \delta^{\prime} \leq 1/2$, then it also holds 
(with the same $\delta^{\prime}$) if $\chi$ is imprimitive, since the $L$--functions of an imprimitive and primitive character 
differ by a finite Euler product, uniformly bounded above and below by $O(1)$ in the strip $1/2 < \sigma \leq 1 + \epsilon$. 
Therefore, for $\chi \in Q(\Cl_{\mfc^2 \mfc_\infty}(F))$ with $\chi \neq \chi_{\frak{c},0}$, we have 
\begin{align*}
L_F(\chi,s) & \ll_{\epsilon, d} \mathfrak{q}(F,\chi,s)^{\delta^{\prime}(1-\sigma)+2\epsilon}\\ 
& \ll_{\epsilon, d} (d_F(|s| + 4)^d)^{\delta^{\prime}(1-\sigma)+2\epsilon}, \quad 1/2 < \sigma \leq 1 + \epsilon
\end{align*}
with $\mathcal{N}_{F/\Q}(\mfc^2) = O(1)$ for all allowable $\mfc$.
Also, from (\ref{torsion}) we have the bound 
\begin{align*}
|Q(\Cl_{\mfc^2 \mfc_\infty}(F))|=|\Cl_{\mfc^2 \mfc_\infty}(F) [2]| \ll_d |\Cl(F)[2]| \ll_{\epsilon, d} d_F^{\delta_d + \epsilon}.
\end{align*}
Then arguing as above we get 
\begin{align*}
B(\sigma + it) \ll_{\epsilon, d} (\sigma - 1/2)^{-d} d_F^{\delta^{\prime}(1 - \sigma) + \epsilon}
(1 + |t|)^{d(\delta^{\prime}(1 - \sigma) + \epsilon)} d_F^{\delta_d + \epsilon}, \quad 1/2 < \sigma \leq 1 + \epsilon.
\end{align*}

Finally, since $g_F(s)=f(s)+B(s)$, we have 
\begin{align*}
g_F(\sigma + it) \ll_{\epsilon, d} d_F^{\delta^{\prime}(1 - \sigma) + \delta_d + 2\epsilon} 
\frac{(1 + |t|)^{d(\delta^{\prime}(1 - \sigma) + \epsilon)}}{(\sigma - 1/2)^{d}},
\quad 1/2 < \sigma \leq 1 + \epsilon.
\end{align*}
\end{proof}

Now, given a pair $(d,G)$, let
\begin{align*}
\mathcal{F}_G^{+} :=\{F/\Q:~ \textrm{$F$ totally real of degree $d$}, ~\textrm{Gal}(F^c/\Q) \cong G \}.
\end{align*}
Define the Dirichlet series 
\begin{align*}
D_{G}^{\textrm{cm}}(s) :=  \sum_{F \in \mathcal{F}_G^{+}} \sum_{[E:F]=2}\frac{1}{|d_E|^s} = \sum_{n=1}^{\infty}\frac{a(n)}{n^{s}},
\end{align*}
where the inner sum is over all totally imaginary quadratic extensions $E/F$ and 
\begin{align*}
a(n):=\# \{(E,F):~ F \in \mathcal{F}_G^{+}, ~ \textrm{$E/F$ totally imaginary quadratic}, ~ |d_E|=n\}.
\end{align*}
Clearly, the Dirichlet series $D_G^{\textrm{cm}}(s)$ 
enumerates all CM fields counted by $N_{2d}^{\text{cm}}(X, G)$. Using the relation 
\begin{align*}
|d_E|=d_F^2\mathcal{N}_{F/\Q}(\frak{D}_{E/F}),
\end{align*}
we have 
\begin{align}\label{eqn:def_phicm}
D_{G}^{\textrm{cm}}(s) =  \sum_{F \in \mathcal{F}_G^{+}} \frac{1}{d_F^{2s}} \sum_{[E:F]=2} \frac{1}{\mathcal{N}_{F/\Q}(\frak{D}_{E/F})^s}
=\sum_{F \in \mathcal{F}_G^{+}} \frac{D_{F, C_2}^{-}(s)}{d_F^{2s}}.
\end{align}

\begin{theorem}\label{thm:kl6}
Assume that Hypothesis \ref{WeakMalle} is true for $(d, G)$ and $\delta^{\prime}$ satisfies 
(\ref{subconvex}). Then the Dirichlet series $D^{\cm}_{G}(s)$ has a meromorphic continuation to the half-plane 
$\mathrm{Re}(s) > \alpha$ where
\begin{align}\label{alphadef}
\alpha=\alpha(\delta_d, M(G), \delta^{\prime}):=\max\left\{\frac{\delta_d + \delta^{\prime} + M(G)}{\delta^{\prime}+2}, \frac{M(G)}{2}\right\} < 1,
\end{align}
with only a single (simple) pole at $s = 1$ given by the convergent series
\begin{align*}
r_d(G):=\sum_{F \in \mathcal{F}_G^{+}}\frac{R_d(F)}{d_F^{2}}=\sum_{F \in \mathcal{F}^{+}_{G}}\frac{\mathrm{Res}_{s=1}\zeta_F(s)}{2^d d_F^2 \zeta_F(2)} > 0.
\end{align*}
Moreover, for $\sigma:=\mathrm{Re}(s) \in (\alpha, 1+ \epsilon]$ and $|t| > 1$, the Dirichlet series $D^{\cm}_{G}(s)$ satisfies the bound
\begin{align}\label{DGbound}
D^{\cm}_{G}(\sigma + it) \ll_{\epsilon, d, G} \frac{(1 + |t|)^{d\delta^{\prime}(1-\sigma) + \epsilon}}{(\sigma - 1/2)^{d}}.
\end{align}

\end{theorem}

\begin{proof} Write 
\begin{align}\label{DGdecomp1}
D^{\cm}_{G}(s) = g(s) + \frac{1}{s-1}h(s),
\end{align}
where
\begin{align*}
g(s):=\sum_{F \in \mathcal{F}_G^{+}}\frac{g_F(s)}{d_F^{2s}} \quad \textrm{and} \quad h(s):=\sum_{F \in \mathcal{F}_G^{+}}\frac{R_d(F)}{d_F^{2s}}.
\end{align*}
Using the estimate  \eqref{eqn:convexity}, we have 
\begin{align*}
%\label{gestimate1}
g(s) \ll_{\epsilon, d} \frac{(1 + |t|)^{d\delta^{\prime}(1-\sigma) + \epsilon}}{(\sigma - 1/2)^{d}}  
 \sum_{F \in \mathcal{F}_G^{+}} d_F^{\delta^{\prime}(1 - \sigma) + \delta_d - 2\sigma + \epsilon}, \quad 1/2 < \sigma \leq 1 + \epsilon.
\end{align*}
Hence, the absolute convergence of the series $g(s)$ is guaranteed by the convergence of the series
\begin{equation}\label{eq:conv1}
\sum_{F \in \mathcal{F}_G^{+}} d_F^{\delta^{\prime}(1 - \sigma) + \delta_d - 2\sigma + \epsilon}.
\end{equation}
Divide the sum over $F$ into intervals $N < d_F \leq 2N$ and let $N$ range over the integer powers of $2$.  
Then using the estimate \eqref{eq:malle}, we see that \eqref{eq:conv1} converges whenever the series
\begin{equation}\label{eq:conv2}
\sum_N N^{\delta^{\prime}(1 - \sigma) + \delta_d - 2\sigma + M(G) + \epsilon}
\end{equation}
converges. The series (\ref{eq:conv2}) converges whenever the exponent is negative, i.e., whenever
$\sigma > \alpha_1$ with $$\alpha_1 := \frac{\delta_d + \delta^{\prime} + M(G)}{\delta^{\prime}+2}$$ (for an appropriate choice of $\epsilon > 0$).
The condition $\alpha_1 < 1$ is equivalent to the condition $\delta_d + M(G) < 2$. Therefore, we see that $g(s)$ is analytic 
for $\sigma > \alpha_1$ with $\alpha_1 < 1$, and that $g(s)$ satisfies the bound 
\begin{align}\label{gest1}
g(\sigma + it) \ll_{\epsilon, d, G} \frac{(1 + |t|)^{d\delta^{\prime}(1-\sigma) + \epsilon}}{(\sigma - 1/2)^{d}}, \quad \alpha_1 < \sigma \leq 1 + \epsilon.
\end{align} 

Next, using the estimate $R_d(F) \ll_{\epsilon, d} d_F^{\epsilon}$ we have 
\begin{align*}
h(s) \ll_{\epsilon,d}\sum_{F \in \mathcal{F}_G^{+}} d_F^{-2\sigma + \epsilon}.
\end{align*}
Then a similar argument using the estimate \eqref{eq:malle} shows that the series 
$h(s)$ converges for $\sigma > \alpha_2$ with $\alpha_2 :=M(G)/2$. The condition $\alpha_2 < 1$ is ensured by 
$\delta_d + M(G) < 2$. Therefore, we see that $(s-1)^{-1}h(s)$ is meromorphic for $\sigma > \alpha_2$ with $\alpha_2 < 1$ 
with only a single (simple) pole at $s=1$ with residue 
\begin{align*}
r_d(G):=\sum_{F \in \mathcal{F}_G^{+}}\frac{R_d(F)}{d_F^{2}},
\end{align*}
and that $(s-1)^{-1}h(s)$ satisfies the bound 
\begin{align}\label{hest1}
\frac{1}{s-1}h(s) \ll_{\epsilon, d, G} 1, \quad \sigma > \alpha_2, \quad |t| > 1. 
\end{align} 
 
Finally, from (\ref{DGdecomp1}) we conclude that $D^{\cm}_{G}(s)$ has a meromorphic continuation to $\sigma > \alpha:=\max\{\alpha_1, \alpha_2\}$ with 
$\alpha < 1$ with only a single (simple) pole at $s=1$ with residue $r_d(G)$. Moreover, from the estimates (\ref{gest1}) and (\ref{hest1}) we see that 
$D^{\cm}_{G}(s)$ satisfies the bound
\begin{align*}
D^{\cm}_{G}(\sigma + it) \ll_{\epsilon, d, G} \frac{(1 + |t|)^{d\delta^{\prime}(1-\sigma) + \epsilon}}{(\sigma - 1/2)^{d}}, 
\quad \alpha < \sigma \leq 1+ \epsilon, \quad |t| > 1.
\end{align*}
\end{proof}

\begin{theorem}\label{residue} \textrm{$(i)$} Under the assumptions of Theorem \ref{thm:kl6}, we have 
\begin{align*}
N_{2d}^{\cm}(X, G) = r_d(G) X + O_{d, \epsilon}(X^{\beta(\delta_d, M(G), \delta^{\prime}) + \epsilon})
\end{align*}
where 
\begin{align*}
\beta(\delta_d, M(G), \delta^{\prime}):= 1 - \frac{1 - \alpha}{1 + d\delta'(1 - \alpha)}
%\beta(\delta_d, M(G), \delta^{\prime}):=\alpha + \frac{d\delta^{\prime}(1-\alpha)^2}{1 + d\delta^{\prime}(1-\alpha)} < 1
\end{align*}
with $\alpha=\alpha(\delta_d, M(G), \delta^{\prime}) < 1$ defined by (\ref{alphadef}).
\vspace{0.05in}

\textrm{$(ii)$} If Hypothesis \ref{WeakMalle} is true for every pair $(d, G)$ where $G$ ranges over all  
transitive subgroups $G \leq  S_d$, then 
\begin{align*}
N_{2d}^{\mathrm{cm}}(X) = \left(\sum_{G \leq S_d}r_d(G)\right)X + O_{d, \epsilon}(X^{\beta(\delta_d, \delta^{\prime}) + \epsilon})
\end{align*}
where 
\begin{equation}\label{betadef}
\beta(\delta_d, \delta^{\prime}):= \max_{G \leq S_d}\beta(\delta_d, M(G), \delta^{\prime}) < 1.
\end{equation}
\end{theorem}

\begin{proof} Fix a smooth function $\phi : [0, 1] \rightarrow [0, 1]$ with $\phi(0) = 1$ and $\phi(1) = 0$. 
Then, for each $Y > 1$, define 
\begin{equation*}
\phi_Y(t) := 
\begin{cases}
1, & \text{if } t \in [0, 1]; \\
\phi(Y(t - 1)), & \text{if } t \in [1, 1 + Y^{-1}]; \\
0 & \text{if } t \geq 1 + Y^{-1}.
\end{cases}
\end{equation*}
Let 
\begin{align*}
\widehat{\phi_{Y}}(s):=\int_{0}^{\infty}\phi_Y(t)t^{s-1}dt, \quad \textrm{Re}(s) > 0
\end{align*} 
be the Mellin transform of $\phi_Y$. Integrating by parts $A \geq 1$ times yields the estimate 
\begin{align}\label{Mellinestimate}
\widehat{\phi_{Y}}(s) \ll Y^{-1}\left(\frac{Y}{1 + |t|}\right)^{A},
\end{align}
valid for all $s$ in any fixed vertical strip $\sigma_0 \leq \textrm{Re}(s) \leq \sigma_1$ with $\sigma_0 > 0$,
and also valid for all real numbers $A \geq 1$ by interpolation.

By construction, and then by Mellin inversion,
we have 
\begin{align*}
%\label{countdecomp}
N_{2d}^{\cm}(X, G) = \sum_{n = 1}^X a(n) \leq \sum_{n=1}^{\infty}a(n)\phi_Y\left(\frac{n}{X}\right)
= \frac{1}{2\pi i } \int_{(1 + \epsilon)}D_G^{\mathrm{cm}}(s)\widehat{\phi_{Y}}(s)X^s ds.
\end{align*}

From the estimate (\ref{DGbound}), we see that 
\begin{align}\label{Dstrips}
D_G^{\mathrm{cm}}(s) \ll (1 + |t|)^{d\delta^{\prime}(1-\sigma)+ \epsilon}
\end{align}
in any vertical strip $1/2 < \alpha + \eta < \sigma \leq 1+ \epsilon$, $|t| > 1$, where the implied constant depends on $\epsilon, d, F$, and $\eta$.
Then using the estimates (\ref{Mellinestimate}) and (\ref{Dstrips}), 
we may shift the contour to $\textrm{Re}(s) = \alpha'$ with $\alpha < \alpha' < 1$ to get 
\begin{align*}
\frac{1}{2\pi i } \int_{(1 + \epsilon)}D_G^{\mathrm{cm}}(s)\widehat{\phi_{Y}}(s)X^sds = \widehat{\phi_{Y}}(1)r_d(G)X 
+ \frac{1}{2\pi i } \int_{(\alpha')}D_G^{\mathrm{cm}}(s)\widehat{\phi_{Y}}(s)X^sds.
\end{align*}
For any $A \geq 1$ we have the estimate
\begin{align*}
\frac{1}{2\pi i } \int_{(\alpha')}D_G^{\mathrm{cm}}(s)\widehat{\phi_{Y}}(s)X^sds \ll 
X^{\alpha'}Y^{-1}\int_{\R}(1 + |t|)^{d\delta^{\prime}(1-\alpha')+ \epsilon}\left(\frac{Y}{1 + |t|}\right)^{A}dt.
\end{align*}
Choose $A=d\delta^{\prime}(1-\alpha')+ 1 + 2\epsilon$. Then 
\begin{align*}
\frac{1}{2\pi i } \int_{(\alpha + \varepsilon)}D_G^{\mathrm{cm}}(s)\widehat{\phi_{Y}}(s)X^sds \ll X^{\alpha'} 
Y^{d\delta^{\prime}(1-\alpha')+ 2\epsilon}.
\end{align*}
Since $\widehat{\phi_{Y}}(1) = 1 + O(Y^{-1})$, we have
\begin{align*}
\widehat{\phi_{Y}}(1)r_d(G)X = r_d(G)X + O(XY^{-1}). 
\end{align*}
Then putting things together, and replacing $2 \epsilon$ by $\epsilon$, we get 
\begin{align}\label{countdecomp2}
N_{2d}^{\text{cm}}(X, G) \leq
\sum_{n=1}^{\infty}a(n)\phi_Y\left(\frac{n}{X}\right) = r_d(G)X + O(XY^{-1}) + O(X^{\alpha'} Y^{d\delta^{\prime}(1-\alpha')+ \epsilon}).
\end{align} 

Similarly, we have 
\begin{equation*}
N_{2d}^{\text{cm}}(X, G) \geq
\sum_{n=1}^{\infty}a(n)\phi_Y\Big(\frac{n}{X - XY^{-1}}\Big), 
\end{equation*}
for which the same estimate in \eqref{countdecomp2} also holds (since we may interchange $X$ and $X - XY^{-1}$
in \eqref{countdecomp2}, within the error terms given there), so that in fact we have
\begin{align*}
N_{2d}^{\mathrm{cm}}(X, G) = r_d(G)X + O(XY^{-1}) + O(X^{\alpha'} Y^{d\delta^{\prime}(1-\alpha^{\prime})+ \epsilon}).
\end{align*}

We optimize (apart from epsilon factors) by choosing 
$\alpha' = \alpha + \epsilon$ and 
$
Y = X^{\frac{1-\alpha}{1 + d\delta^{\prime}(1-\alpha)}}
$
, so as to obtain for each $\epsilon > 0$ that
\begin{align*}
N_{2d}^{\mathrm{cm}}(X, G) = r_d(G)X + O_{\epsilon}(X^{\beta(\delta_d, M(G), \delta^{\prime}) + \epsilon}),
\end{align*}
where 
\begin{align*}
\beta(\delta_d, M(G), \delta^{\prime}):= 1 - \frac{1 - \alpha}{1 + d\delta'(1 - \alpha)}.
\end{align*}
This proves part $(i)$. Part $(ii)$ follows by summing the asymptotic formula in $(i)$ over all transitive subgroups $G \leq S_d$.
\end{proof}

\subsection{Upper bounds  for $N_{2d}^{\neg \Weyl}(X, G)$} Let $K$ be a number field of degree $d$ with $\Gal(K^c/\Q) \cong G$, and 
let $L$ be a quadratic extension of $K$. Then $\Gal(L^c/\Q)$ embeds as a subgroup of the wreath product $C_2 \wr G$ (see Proposition \ref{embed}). Clearly, we have 
\begin{align*}
N_{2d}^{\neg \Weyl}(X, G) \ll Y(X,G),
\end{align*}
where 
\begin{align}\label{eqn:def_Y}
Y(X,G):=\# \{ L/K: ~ \Gal(L^c/\Q) \ncong C_2 \wr G, ~ \Gal(K^c/\Q) \cong G, ~ [L:K]=2, |d_L| \leq X \}.
\end{align}
Therefore, it suffices to give an upper bound for $Y(X,G)$. 

The extensions counted by $Y(X,G)$ are distinguished by the following fact:
for each prime $p$ unramified in $K/\Q$ but ramified in $L/K$ (so that $p \mid d_L$),
we must in fact have $p^2 \mid d_L$ (see \cite[Lemma 4]{Klu12}). 

Let
\begin{align*}
\mathcal{K}_G(X^{1/2}) :=\{K/\Q:~ \textrm{Gal}(K^c/\Q) \cong G, ~ |d_K| \leq X^{1/2} \}.
\end{align*}
As in \cite[p. 9-10]{Klu12}, we have the bound
\begin{align}\label{eq:kluners910}
Y(X,G) \leq \sum_{K \in \mathcal{K}_G(X^{1/2})}O_{\epsilon, d}\left(\frac{X^{\frac12 + \epsilon}}{|d_K|^2}|\mathrm{Cl}(K)[2]|\right).
\end{align}
We briefly recall the proof. Each $L$ counted in \eqref{eqn:def_Y} satisfies $$d_L = d_K^2 \mathcal{N}_{K/\Q}(\mathfrak{D}_{L/K})$$ with
$\mathcal{N}_{K/\Q}(\mathfrak{D}_{L/K}) = a b^2$, where $a$ is only divisible by primes dividing $d_K$. Since each such prime can only
divide $a$ with bounded multiplicity, the problem is reduced to proving (for each positive integer $n$)
that the number of quadratic extensions $L/K$ with
$\mathcal{N}_{K/\Q}(\mathfrak{D}_{L/K}) = n$ is $O_{d, \epsilon}(|\Cl(K)[2]| n^{\epsilon})$, and this is done by bounding the $2$-torsion
in the relevant ray class group. 

Continuing then, applying the bound (\ref{torsion}) to \eqref{eq:kluners910} gives
\begin{equation*}
N_{2d}^{\neg \Weyl}(X, G) \ll_{\epsilon, d} 
 X^{\frac{1}{2} + \epsilon} \sum_{K \in \mathcal{K}_G(X^{1/2})} |d_K|^{-1 + \delta_d}.
\end{equation*}
Again, divide the sum over $K$ into intervals with $N < |d_K| \leq 2N$ and let 
$N$ range over the integer powers of $2$. Then applying the estimate \eqref{eq:malle} gives
\begin{equation*}
N_{2d}^{\neg \Weyl}(X, G) \ll_{\epsilon, d} 
 X^{\frac12 + \epsilon} \sum_{N} N^{-1 + \delta_d + M(G)}.
 \end{equation*}
If $\delta_d + M(G) \leq 1$ then
\begin{align}\label{notWeyl1}
N_{2d}^{\neg \Weyl}(X, G) \ll_{\epsilon, d}  X^{\frac12 + \epsilon},
\end{align}
while if $\delta_d + M(G) > 1$ then 
\begin{equation}\label{notWeyl2}
N_{2d}^{\neg \Weyl}(X, G) \ll_{\epsilon, d} 
 X^{\frac12 + \frac{-1 + \delta_d + M(G)}{2} + \epsilon}.
\end{equation}
The exponent in (\ref{notWeyl2}) is less than $1$ (for an appropriate choice of $\epsilon > 0$) provided that
$\delta_d + M(G) < 2$.

\subsection{Proof of Theorem \ref{main_thm}} Using Theorem \ref{residue} and estimates \eqref{notWeyl1} and \eqref{notWeyl2}, we have 
\begin{equation*}
%\label{eqn:main_thm}
\frac{ N_{2d}^{\Weyl}(X, G)}{N_{2d}^{\cm}(X, G)} = \frac{N_{2d}^{\cm}(X, G) - N_{2d}^{\neg \Weyl}(X, G)}{N_{2d}^{\cm}(X, G)} 
= 1 + O_{d, G, \epsilon}(X^{-C_1(\delta_d,M(G)) + \epsilon}),
\end{equation*}
where 
\begin{align*}
C_1(\delta_d,M(G)):= 
%1 - \gamma(\delta_d,M(G))=
\begin{cases}
1/2, & \textrm{if $\delta_d + M(G) \leq 1$}\\
\displaystyle 1- \frac{\delta_d + M(G)}{2}, &  \textrm{if $1 < \delta_d + M(G) < 2$}.
\end{cases}
\end{align*}
This proves Theorem \ref{main_thm}. 

\subsection{Proof of Theorem \ref{WeylDensity}}\label{proofsection} 
As above we have
\begin{align*}
%\label{a1}
\frac{N_{2d}^{\Weyl}(X, G)}{N_{2d}^{\cm}(X)} =  
\frac{N_{2d}^{\cm}(X, G) - N_{2d}^{\neg \Weyl}(X, G)}{N_{2d}^{\cm}(X)} = \frac{N_{2d}^{\cm}(X, G)}{N_{2d}^{\cm}(X)} 
+ O_{d, G, \epsilon}(X^{-C_1(\delta_d, M(G))+\epsilon}).
%\frac{N_{2d}^{\cm}(X, G)}{N_{2d}^{\cm}(X)} + O_{d, G}(X^{-\beta})
\end{align*}
Also by Theorem \ref{residue} we have
\begin{align*}
%\label{a2}
\frac{N_{2d}^{\cm}(X, G)}{N_{2d}^{\cm}(X)} = \frac{r_d(G)}{\sum_{G \leq S_d}r_d(G)} + 
O_{d, \epsilon}(X^{-1 + \beta(\delta_d, \delta^{\prime})})
\end{align*}
so that
\begin{align}\label{last}
\frac{N_{2d}^{\Weyl}(X, G)}{N_{2d}^{\cm}(X)} = \frac{r_d(G)}{\sum_{G \leq S_d}r_d(G)}
+ O_{d,G,\epsilon}(X^{-C_2(\delta_d, M(G), \delta^{\prime}) + \epsilon}),
\end{align}
where 
\begin{equation}\label{eq:def_C2}
C_2(\delta_d, M(G), \delta^{\prime}):=\min \{C_1(\delta_d, M(G)), 
1 - \beta(\delta_d, \delta') \} > 0.
\end{equation}
This proves (\ref{full1}). To prove (\ref{full2}), we sum over all $G \leq S_d$ in (\ref{last}) to get 
\begin{align*}
\frac{N_{2d}^{\Weyl}(X)}{N_{2d}^{\cm}(X)} = 1 + O_{d,\epsilon}(X^{-C_3(\delta_d, \delta^{\prime}) + \epsilon}),
\end{align*} 
where 
\begin{equation}\label{eq:def_C3}
C_3(\delta_d,\delta^{\prime}):= \min_{G \leq S_d}C_2(\delta_d, M(G), \delta^{\prime}) > 0.
\end{equation}
This proves Theorem \ref{WeylDensity}. 
\qed

\section{Proof of Theorem \ref{ColmezWeyl}}\label{wreathsection}

In this section we review some basic facts about wreath products of groups, discuss the 
structure of Galois groups of CM fields, and prove Theorem \ref{ColmezWeyl}. 

\subsection{Wreath products}\label{wreath} 
We begin by reviewing some basic 
facts about wreath products of groups (see e.g. \cite{DM96}). 
Let $H$ and $K$ be groups, and suppose that $\theta: H \rightarrow \op{Aut}(K)$ 
is a homomorphism, where we write $\theta(h) = \theta_h$. This gives a (left) group 
action of $H$ on $K$ defined by $(h,k) \mapsto \theta_h(k)$.
Recall that the semidirect product of $K$ and $H$ with respect to $\theta$ is the group
\begin{align*}
K \rtimes_{\theta} H := \{(k, h) \suchthat k \in K, h \in H \},
\end{align*}
where the group operation is defined by
\begin{align*}
(k_1, h_1) (k_2, h_2) := (k_1 \theta_{h_1}(k_2), h_1 h_2).
\end{align*}
When understood, we suppress $\theta$ in our notation for the semidirect product.

Now, let $\Omega$ be an arbitrary set, and let $K^{\Omega}$ denote the set of all functions $f: \Omega \rightarrow K$. 
Pointwise multiplication of functions gives $K^{\Omega}$ the structure of a group. A (left) group action of $H$ on $\Omega$ gives a homomorphism
\begin{align*}
\theta: H &\longrightarrow \op{Aut}(K^{\Omega})\\
h &\longmapsto \theta_{h}
\end{align*}
defined by $\theta_{h}(f)(\omega):= f(h^{-1} \cdot \omega)$ for every $\omega \in \Omega$ and every $f \in K^{\Omega}$. In turn, 
this gives a (left) group action of $H$ on $K^{\Omega}$ defined by $(h,f) \mapsto \theta_h(f)$. The 
\textit{wreath product} of $K$ and $H$ with respect to $\theta$ is defined by 
\begin{align*}
K \wr_{\Omega} H := K^{\Omega} \rtimes_{\theta} H.
\end{align*}

When the set $\Omega = \{ \omega_1, \dots \omega_n \}$ is finite, 
it is customary to identify $K^{\Omega}$ with the direct product $K^n$ via the isomorphism 
$f \mapsto (f(\omega_1), \dots, f(\omega_n))$. In particular, if $\Omega = \{ 1, \dots, n \}$ and $H \leq S_n$ 
is a group of permutations, then we have a (left) group action of $H$ on $\Omega$ in the usual way, 
and the corresponding action of $H$ on $K^n$ is by permutation of the components, i.e., 
if $\sigma \in H \leq S_n$ and $x = (x_1, \dots, x_n) \in K^n$, then
\begin{align*}
\sigma \circ x := (x_{\sigma^{-1}(1)}, \dots, x_{\sigma^{-1}(n)}),
\end{align*}
and in this case we write $K \wr H$ instead of $K \wr_{\{ 1, \dots, n \}} H$. 

With this notation, if $G \leq S_d$ is a transitive subgroup, then 
the wreath product $C_2 \wr G$ from the introduction is given by  
$$C_2 \wr G = C_2 \wr_{\{ 1, \dots, d \}} G = C_2^d \rtimes  G.$$ 
The wreath product determines a short exact sequence 
\begin{equation*}
%\label{wreath}
\begin{tikzcd}
1 \arrow{r} & C_2^d \arrow{r} &  C_2 \wr G \arrow{r} & G \arrow{r} & 1.
\end{tikzcd}
\end{equation*}

\subsection{Galois groups of CM fields} We next discuss the structure of Galois groups of CM fields.

\begin{proposition}\label{embed} Let $K$ be a number field of degree $d$ with $\mathrm{Gal}(K^c/\Q) \cong G \leq S_d$, 
and let $L$ be a quadratic extension of $K$. Then $\mathrm{Gal}(L^c/\Q)$ embeds as a subgroup of 
the wreath product $C_2 \wr G$.
\end{proposition}

\begin{proof} 

Choose a primitive element $\alpha_1$ with $K = \Q(\alpha_1)$ and let 
$\alpha_1, \dots, \alpha_d$ be its conjugates, so that
$K^c = \Q(\alpha_1, \dots, \alpha_d)$ and $L^c=\Q(\sqrt{\alpha_1}, \ldots, \sqrt{\alpha_d})$.
For each $g \in \mathrm{Gal}(L^c/\Q)$ and $i \in \{1, \dots, d\}$, we have
\begin{equation}\label{eq:wreath1}
g(\alpha_i)=\alpha_j, \ \ \ g(\sqrt{\alpha_i}) = \pm \sqrt{\alpha_j}
\end{equation}
for some $j \in \{1, \dots, d \}$ and choice of sign $\pm$. We define a function
\begin{align*}
\phi : \mathrm{Gal}(L^c/\Q) & \longrightarrow C_2 \wr G = \{\pm 1\}^d  \rtimes G\\
g & \longmapsto (x_g, \sigma_g),
\end{align*}
where (matching \eqref{eq:wreath1}) $\sigma_g(i) = j$ and $x_g:=(x_{g,1}, \ldots, x_{g,d}) \in \{\pm 1\}^{d}$ is the vector 
whose $j$-th component is given by  
\begin{align*}
x_{g, j} := \frac{g(\sqrt{\alpha_i})}{\sqrt{\alpha_j}}=\frac{g\big(\sqrt{\alpha_{\sigma_g^{-1}(j)}}\big)}{\sqrt{\alpha_j}}.
\end{align*}
In particular, we have 
\begin{align*}
g(\sqrt{\alpha_i})=x_{g, \sigma_g(i)}\sqrt{\alpha_{\sigma_g(i)}}
\end{align*}
for $i \in \{1, \ldots, d\}$.
%%=\frac{g\big(\sqrt{\alpha_{\sigma_g^{-1}(j)}}\big)}{\sqrt{\alpha_j}}$ 
%for each $j \in \{1, \ldots, d\}$.

The data of $x_g$ and $\sigma_g$ determines $g(\sqrt{\alpha_i})$ for each $i$, and hence it determines
$g$, so that $\phi$ is injective.

We next prove that $\phi$ is a homomorphism. Let $g, h \in \mathrm{Gal}(L^c/\Q)$. 
%and suppose that $\tilde{g}_{|_{F^c}}(\alpha_j)=\alpha_k$. 
Then by definition of the wreath product, we have 
\begin{align*}
\phi(gh) = \phi(g)\phi(h)  \quad
%& \Longleftrightarrow \quad (x_{gh}, \sigma_{gh}) = (x_{g}, \sigma_{g})(x_h, \sigma_{h})\\ 
\Longleftrightarrow  \quad (x_{gh}, \sigma_{gh})=(x_{g}(\sigma_{g} \circ x_h), \sigma_{g}\sigma_{h})
\end{align*}
where 
\begin{align*}
\sigma_{g} \circ x_h := (x_{h,\sigma_{g}^{-1}(1)}, \ldots, x_{h,\sigma_{g}^{-1}(d)}). 
\end{align*}

By the isomorphism $\textrm{Gal}(K^c/\Q) \cong G$, we have 
$\sigma_{gh}=\sigma_{g}\sigma_{h}$. Thus, it remains to prove that 
\begin{align}\label{equal}
x_{gh}=x_{g}(\sigma_{g} \circ x_h).
\end{align}
Since the $\sigma_{gh}(i)$-th component of $\sigma_g \circ x_h$ is given by 
\begin{align*}
x_{h,\sigma_{g}^{-1}(\sigma_{gh}(i))}=x_{h, \sigma_h(i)}, 
\end{align*}
we see that (\ref{equal}) is equivalent to 
\begin{align*}
x_{gh, \sigma_{gh}(i)}=x_{g, \sigma_{gh}(i)}x_{h,\sigma_h(i)}
\end{align*}
for $i \in \{1, \ldots, d\}$. We compute 
\begin{align*}
(gh)(\sqrt{\alpha_i})&=g(h(\sqrt{\alpha_i}))\\
&=g(x_{h,\sigma_h(i)}\sqrt{\alpha_{\sigma_h(i)}})\\
&=x_{h,\sigma_h(i)}g(\sqrt{\alpha_{\sigma_h(i)}})\\
&=x_{h,\sigma_h(i)}x_{g,\sigma_g(\sigma_{h}(i))}\sqrt{\alpha_{\sigma_g(\sigma_{h}(i))}}\\
&=x_{h,\sigma_h(i)}x_{g,\sigma_{gh}(i)}\sqrt{\alpha_{\sigma_{gh}(i)}},
\end{align*}
and thus 
\begin{align*}
x_{gh,\sigma_{gh}(i)}:=\frac{(gh)(\sqrt{\alpha_i})}{\sqrt{\alpha_{\sigma_{gh}(i)}}}=x_{h,\sigma_h(i)}x_{g,\sigma_{gh}(i)}.
\end{align*}
This completes the proof.
\end{proof}

\begin{proposition}\label{non-abelian} 
Let $d \geq 2$ and suppose that $G$ is a transitive subgroup of $S_d$. Then the wreath product $C_2 \wr G$ is non-abelian. In particular, 
if $E$ is a $G$-Weyl CM field of degree $2d \geq 4$, then $E/\Q$ is non-abelian.
\end{proposition}

\begin{proof} As we have seen, the elements of
the wreath product $C_2 \wr G$ take the form $(x, \sigma)$ with $x = (x_1, \dots, x_d) \in C_2^d$ and $\sigma \in G$ 
a permutation of the set $\{1, \dots, d \}$, with multiplication given by
\begin{align*}
(x, \sigma) (y, \tau) = (x (\sigma \circ y), \sigma \tau)
\end{align*}
where 
\begin{align*}
\sigma \circ y := (y_{\sigma^{-1}(1)}, \ldots, y_{\sigma^{-1}(d)}).
\end{align*}
It now suffices to exhibit two elements which do not commute; for example, choose $x = y = (-1, 1, \dots, 1)$,
$\tau = \op{id}$, and any $\sigma$ such that $\sigma^{-1}(-1) = d$ (the existence of which is ensured by the transitivity of $G$). 

Finally, if $E$ is a $G$-Weyl CM field of degree $2d \geq 4$, then we have 
$\textrm{Gal}(E^c/\Q) \cong C_2 \wr G$, so that $E/\Q$ is non-abelian.
\end{proof}

Now, let $E$ be a CM field of degree $2d$ with maximal totally real subfield $F$, 
and let $G$ be the transitive subgroup of $S_d$ with $\mathrm{Gal}(F^c/\Q) \cong G$.
Choose a primitive element $\alpha_1$ with $F = \Q(\alpha_1)$ and let 
$\alpha_1, \dots, \alpha_d$ be its conjugates, so that
$F^c = \Q(\alpha_1, \dots, \alpha_d)$ and
$E^c=\Q(\sqrt{-\alpha_1}, \ldots, \sqrt{-\alpha_d})$.
For each $g \in \mathrm{Gal}(E^c/F^c)$ and $i \in \{1, \dots, n\}$, we have
\begin{equation*}
%\label{eq:wreath2}
g(\alpha_i)=\alpha_i, \ \ \ g(\sqrt{-\alpha_i}) = \pm \sqrt{-\alpha_i}
\end{equation*}
for some choice of sign $\pm$. We define a function
\begin{align*}
\psi : \mathrm{Gal}(E^c/F^c) & \longrightarrow \{\pm 1\}^d \\
g & \longmapsto y_g,
\end{align*}
where $y_g:=(y_{g,1}, \ldots, y_{g,d}) \in \{\pm 1\}^d$ is the vector whose $i$-th component is given by 
%(matching \eqref{eq:wreath2})
\begin{align*}
y_{g, i} := \frac{g(\sqrt{-\alpha_i})}{\sqrt{-\alpha_i}}.
\end{align*}
Then arguing as in Proposition \ref{embed}, we see that $\psi$ is an injective homomorphism.

By Galois theory, we have the short exact sequence
\begin{equation}\label{galoisexact}
\begin{tikzcd}
1 \arrow{r} &\Gal(E^c/F^c) \arrow{r} &\Gal(E^c/\Q) \arrow{r} &\Gal(F^c/\Q) \arrow{r} & 1.
\end{tikzcd}
\end{equation}
The injective homomorphism $\psi$ 
implies that $\Gal(E^c/F^c) \cong C_2^v$ for some $1 \leq v \leq d$. Then from (\ref{galoisexact}) we 
get a short exact sequence
\begin{equation}\label{imprimitivity}
\begin{tikzcd}
1 \arrow{r} &C_2^v \arrow{r} & \Gal(E^c/\Q) \arrow{r} &  G \arrow{r} & 1.
\end{tikzcd}
\end{equation}
The exact sequence (\ref{imprimitivity}) is called the \textit{imprimitivity sequence} for $\Gal(E^c/\Q)$ 
(see \cite[p. 4]{Dod84}). 

Recall that if $\Phi$ is a CM type for $E$, then the \textit{reflex field} associated to the CM pair $(E, \Phi)$ is the field
\begin{align*}
E_{\Phi}:=  \Q(\{ \op{Tr}_{\Phi}(a) \suchthat a \in E \}),
\end{align*}
where $$\op{Tr}_{\Phi}(a):= \sum \limits_{\phi \in \Phi}\phi(a)$$ 
is the \textit{type trace} of $a \in E$. The Reflex Degree Theorem of Dodson \cite[p. 5]{Dod84} states that the reflex degree $[E_{\Phi}:\Q]$ is related 
to imprimitivity sequences in the following way: if $G \leq S_d$ is the transitive subgroup such that $\Gal(F^c/\Q) \cong G$, and 
\begin{equation}\label{is2}
\begin{tikzcd}
1 \arrow{r} &C_2^v \arrow{r} &\Gal(E^c/\Q) \arrow{r} & G \arrow{r}  & 1
\end{tikzcd}
\end{equation}
is the imprimitivity sequence for $\Gal(E^c/\Q)$, then there is a subgroup $S$ of $G$ such that
\begin{align}\label{ReflexDegree}
[E_{\Phi}:\Q] = 2^v [G:S].
\end{align}
The subgroup $S$ is defined in \cite[p. 5]{Dod84} (the so-called \textit{splitting subgroup}), 
although we do not need an explicit description for our purposes. 

With these preliminaries, we now proceed to the proof of Theorem \ref{ColmezWeyl}.

\subsection{Proof of Theorem \ref{ColmezWeyl}} It is known that the absolute Galois group 
$\Gal(\overline{\Q}/\Q)$ acts on the set of CM types $\Phi(E)$ of $E$. 
Importantly, one can prove that the size of the $\Gal(\overline{\Q}/\Q)$-orbit of a CM type $\Phi$ equals the degree of the reflex field 
$E_{\Phi}$ over $\Q$; that is (see \cite[Proposition 6.3]{BSM16}),
\begin{align*}
[E_{\Phi}:\Q] = \# (\Gal(\overline{\Q}/\Q) \cdot \Phi).
\end{align*}
Since there are exactly $2^d$ CM types in $\Phi(E)$, 
this shows that 
\begin{align}\label{reflex1}
[E_{\Phi}:\Q] \leq 2^d, 
\end{align}
and moreover, that the action of 
$\Gal(\overline{\Q}/\Q)$ on $\Phi(E)$ is transitive if and only if $[E_{\Phi}:\Q] = 2^d$.

Now, suppose that $E$ is a $G$-Weyl CM field. Then 
$$\Gal(E^c/\Q) \cong C_2 \wr G = C_2^d \rtimes  G,$$ and in particular, we have $|\Gal(E^c/\Q)| = 2^d |G|$. On the other hand, 
by the imprimitivity sequence (\ref{is2}) we have $G \cong \Gal(E^c/\Q)/C_2^{v}$, so that 
$|\Gal(E^c/\Q)| = 2^v |G|$. Hence $v=d$, and it follows from (\ref{ReflexDegree}) that 
\begin{align}\label{reflex2}
[E_{\Phi}:\Q] = 2^d [G:S] \geq 2^d.
\end{align}
From inequalities (\ref{reflex1}) and (\ref{reflex2}), we conclude that $[E_{\Phi}:\Q]=2^d$, and thus 
the action of $\Gal(\overline{\Q}/\Q)$ on $\Phi(E)$ is transitive. 

Finally, in \cite[Proposition 5.1]{BSM16}, it is 
shown as a consequence of the recently proved averaged Colmez conjecture \cite{AGHM15, YZ15} that if the action of 
$\Gal(\overline{\Q}/\Q)$ on $\Phi(E)$ is transitive, then the Colmez conjecture is true for $E$ and takes the form (\ref{nice}). 
This proves Theorem \ref{ColmezWeyl}. \qed

\section{Some known cases of Hypothesis \ref{WeakMalle}}\label{known}

In this section, we give a table which lists some known cases of Hypothesis \ref{WeakMalle}. We also give  
a table that lists cases of Hypothesis \ref{WeakMalle} which would follow from a sufficiently strong $2$-torsion exponent $\delta_d$.  

For $d \geq 6$, the lists are extracted from the tables in \cite{Du17}; in particular, as Dummit notes, 
the labeling of the transitive subgroups 
is the standard one originally given by Conway, Hulpke, and McKay \cite{CHM98}. 
For simplicity, when summarizing results in the tables, we sometimes state upper bounds which are weaker than what is known. 

For a transitive subgroup $G \leq S_d$, Table \ref{Table1} gives a list of general pairs $(d,G)$ for which Hypothesis \ref{WeakMalle} is known
to hold. In each case, the upper bound in the Malle conjecture (\ref{wmc}) is known, and we may take $\delta_d = 1/2$.
The table does not necessarily contain a complete list of all known results, and it should be possible to obtain additional cases of 
Hypothesis \ref{WeakMalle}. Among other possibilities, Wang informs us that her methods can handle additional cases such
as $d = 9$, $G = S_3 \times S_3$, and Mehta \cite{Meh17} is presently extending the results of \cite{Klu06} to Frobenius groups.
 
We also note that when $G$ satisfies Hypothesis \ref{WeakMalle}, so does $C_2 \wr G$ by the argument of Kl\"uners
\cite{Klu12} which we are adapting.

\small
\begin{table}[H]
\caption{General pairs $(d, G)$ for which Hypothesis \ref{WeakMalle} holds.} \label{Table1} 
\vspace{0.10in}
{\tabulinesep=1.0mm
\begin{tabu}{ |c|c|c| }
  %\hline
  %\multicolumn{2}{|c|}{Pairs $(d, G)$ for which Hypothesis \ref{WeakMalle} holds} \\
  \hline
 $(d, G)$ & Reference & Upper bound $N_d(X, G) \ll X^{M(G)}$ \\ \hline
 $d\geq 1$ and $G$ abelian & \cite{Mak85} & $X^{\frac{1}{|G|(1 - 1/\ell)}+ \epsilon}$, \ $\ell$ the smallest prime divisor of $|G|$ \\ \hline
$d = \ell$ prime, $G = D_\ell$ & \cite{Klu06, CT17} & $X^{\frac{3}{\ell - 1} - \frac{1}{\ell (\ell - 1)} + \epsilon}$ \\ \hline
$d \geq 1$ and $G$ a $p$-group & \cite{KM04} & $X^{1 + \epsilon}$ \\ \hline
$d \geq 5$ and $|G|=d$ & \cite{EV06}  &  $X^{\frac{3}{8} + \epsilon}$ \\ \hline
$d = 3$, any $G \leq S_3$ transitive & \cite{DH71} & $X^1$ \\ \hline
$d = 4$, any $G \leq S_4$ transitive & \cite{CDO02, Bha05} & $X^1$ \\ \hline
$d = 5$, any $G \leq S_5$ transitive & \cite{Bha10} & $X^1$ \\ \hline
$d = 3|A|$, $S_3 \times A$ with any $A$ abelian & \cite{Wan17} & $X^{1/|A|}$ \\ \hline
$d = 4|A|$, $S_4 \times A$ with any $A$ abelian & \cite{Wan17} & $X^{1/|A|}$ \\ \hline
\end{tabu}}
%\caption{Pairs $(d, G)$ for which the upper bound in (\ref{wmc}) holds. All of these pairs satisfy 
%Hypothesis \ref{WeakMalle} with $\delta = 1/2$.}
\end{table}
\normalsize

\begin{remark}
As observed previously, the condition $|G|=d$ 
is equivalent to all number fields counted by $N_{d}(X,G)$ being Galois over $\Q$. This case follows from \cite[Proposition 1.3]{EV06}. 
\end{remark}

Table \ref{Table4} lists specific pairs $(d, G)$ with $6 \leq d \leq 8$, for which an upper bound 
$N_d(X, G) \ll X^{M(G)}$ is known for some $M(G) < 2$, but such that $\delta_d + M(G) > 2$. The last column lists a range of $2$-torsion exponents 
which would suffice for $\delta_d +  M(G) < 2$ to hold. 

The results were obtained by Dummit \cite{Du17}.

%\small
\begin{table}[H]
\caption{Specific pairs $(d,G)$ for which Hypothesis \ref{WeakMalle} holds for any $2$-torsion exponent 
$\delta_d$ in the specified range.}\label{Table4}
\vspace{0.10in}
{\tabulinesep=1.0mm
\begin{tabu}{ |c|c|c|c|c| }
\hline
Label \# & Order of group & Isomorphic to & Upper bound $N_d(X, G) \ll X^{M(G)}$ & Range of $\delta_d$\\
\hline
\hline
\multicolumn{5}{|c|}{Transitive subgroups of $S_6$ satisfying $N_6(X, G) \ll X^{M(G)}$ with $M(G) < 2$ $(d = 6)$} \\
\hline
6T5 & 18 & $F_{18}$ & $X^{7/4 + \epsilon}$ & $\delta_6 < \frac{1}{4}$\\ \hline
6T12 & 60 & $A_5$ & $X^{8/5 + \epsilon}$ & $\delta_6 < \frac{2}{5}$\\ \hline
6T14 & 120 & $S_5$ & $X^{19/10 + \epsilon}$ & $\delta_6 < \frac{1}{10}$\\ \hline
6T15 & 360 & $A_6$ & $X^{19/10 + \epsilon}$ & $\delta_6 < \frac{1}{10}$\\ \hline
\hline
\multicolumn{5}{|c|}{Transitive subgroups of $S_7$ satisfying $N_7(X, G) \ll X^{M(G)}$ with $M(G) < 2$ $(d = 7)$} \\ \hline
7T2 & 14 & $D_7$ & $X^{19/12 + \epsilon}$ & $\delta_7 < \frac{5}{12}$ \\ \hline
7T3 & 21 & $F_{21}$ & $X^{7/4 + \epsilon}$ & $\delta_7 < \frac{1}{4}$ \\ \hline
7T5 & 168 & $PSL_2(\mathbb{F}_7)$ & $X^{11/6 + \epsilon}$ & $\delta_7 < \frac{1}{6}$ \\ \hline \hline
\multicolumn{5}{|c|}{Transitive subgroups of $S_8$ satisfying $N_8(X, G) \ll X^{M(G)}$ with $M(G) < 2$ $(d = 8)$} \\ \hline
8T25 & 56 & $F_{56}$ & $X^{27/14 + \epsilon}$ & $\delta_8 < \frac{1}{14}$ \\ \hline
%\multicolumn{5}{|c|}{Transitive subgroups of $S_9$ satisfying $N_9(X, G) \ll X^{M(G)}$ with $M(G) < 2$ $(d = 9)$} \\ \hline
%9T4 & 18 & $S_3 \times C_3$ & $X^{23/12 + \epsilon}$ & $\delta_9 < \frac{1}{12}$ \\ \hline
\end{tabu}}
%\caption{This table lists pairs $(d, G)$ for which an upper bound 
%$N_d(X, G) \ll X^{M + \varepsilon}$ is known to hold with some $M < 2$ and the improvement on 
%$\delta$ on the 2-torsion bound $|\op{Cl}(K)[2]| \ll_{\varepsilon, d} |d_K|^{\delta + \varepsilon}$ that 
%would be needed for Hypothesis \ref{WeakMalle} to hold for the given $(d, G)$.}
\end{table}
%\normalsize

\section{Acknowledgments}

We would like to thank Evan Dummit, Wei-Lun Tsai, Jiuya Wang, and Matt Young for helpful comments.

Barquero-Sanchez and Masri's work was partially supported by the NSF Grants DMS-1162535 and DMS-1460766, and by the Simons Foundation Grant \#421991.

Thorne's work was partially supported by the National Security Agency under a Young Investigator Grant. Part of his work was done 
at the Mathematical Sciences Research Institute in Berkeley, CA in Spring 2017, supported by NSF Grant DMS-1440140.

\end{document}